\documentclass[12pt,reqno]{article}
\usepackage{amsmath,amsthm,amsfonts,amssymb,amscd,amsbsy}
\usepackage{graphics}

\usepackage{color}

\begin{document}

\theoremstyle{plain}
\newtheorem{thm}{Theorem}[section]
\newtheorem{theorem}[thm]{Theorem}
\newtheorem{main theorem}[thm]{Main Theorem}
\newtheorem{lemma}[thm]{Lemma}
\newtheorem{corollary}[thm]{Corollary}
\newtheorem{proposition}[thm]{Proposition}

\theoremstyle{definition}
\newtheorem{notation}[thm]{Notation}
\newtheorem{claim}[thm]{Claim}
\newtheorem{remark}[thm]{Remark}
\newtheorem{remarks}[thm]{Remarks}
\newtheorem{problem}[thm]{Problem}
\newtheorem{conjecture}[thm]{Conjecture}
\newtheorem{definition}[thm]{Definition}
\newtheorem{example}[thm]{Example}

\newcommand{\Max}{{\rm Max \ }}
\newcommand{\sA}{{\mathcal A}}
\newcommand{\sB}{{\mathcal B}}
\newcommand{\sC}{{\mathcal C}}
\newcommand{\sD}{{\mathcal D}}
\newcommand{\sE}{{\mathcal E}}
\newcommand{\sF}{{\mathcal F}}
\newcommand{\sG}{{\mathcal G}}
\newcommand{\sH}{{\mathcal H}}
\newcommand{\sI}{{\mathcal I}}
\newcommand{\sJ}{{\mathcal J}}
\newcommand{\sK}{{\mathcal K}}
\newcommand{\sL}{{\mathcal L}}
\newcommand{\sM}{{\mathcal M}}
\newcommand{\sN}{{\mathcal N}}
\newcommand{\sO}{{\mathcal O}}
\newcommand{\sP}{{\mathcal P}}
\newcommand{\sQ}{{\mathcal Q}}
\newcommand{\sR}{{\mathcal R}}
\newcommand{\sS}{{\mathcal S}}
\newcommand{\sT}{{\mathcal T}}
\newcommand{\sU}{{\mathcal U}}
\newcommand{\sV}{{\mathcal V}}
\newcommand{\sW}{{\mathcal W}}
\newcommand{\sX}{{\mathcal X}}
\newcommand{\sY}{{\mathcal Y}}
\newcommand{\sZ}{{\mathcal Z}}
% Sonderbuchstaben mit Doppellinie
\newcommand{\A}{{\mathbb A}}
\newcommand{\B}{{\mathbb B}}
\newcommand{\C}{{\mathbb C}}
\newcommand{\D}{{\mathbb D}}
\newcommand{\E}{{\mathbb E}}
\newcommand{\F}{{\mathbb F}}
\newcommand{\G}{{\mathbb G}}
\newcommand{\HH}{{\mathbb H}}
\newcommand{\I}{{\mathbb I}}
\newcommand{\J}{{\mathbb J}}
\newcommand{\M}{{\mathbb M}}
\newcommand{\N}{{\mathbb N}}
\renewcommand{\P}{{\mathbb P}}
\newcommand{\Q}{{\mathbb Q}}
\newcommand{\re}{{\mathbb R}}
\newcommand{\R}{{\mathbb R}}
\newcommand{\T}{{\mathbb T}}
\newcommand{\U}{{\mathbb U}}
\newcommand{\V}{{\mathbb V}}
\newcommand{\W}{{\mathbb W}}
\newcommand{\X}{{\mathbb X}}
\newcommand{\Y}{{\mathbb Y}}
\newcommand{\Z}{{\mathbb Z}}
\newcommand{\la}{{\lambda}}
\newcommand{\al}{{\alpha}}
\newcommand{\be}{{\beta}}

\newcommand{\ve}{{\varepsilon}}
\newcommand{\vr}{{\varphi}}

\def\ga{{\gamma}}
\def\vr{{\varphi}}
\def\la{{\lambda}}
\def\al{{\alpha}}
\def\be{{\beta}}
\def\ve{{\varepsilon}}
\def\lv{\left\vert}
\def\rv{\right\vert}

\def\sen{\operatorname{{sen}}}
\def\tr{\operatorname{{tr}}}

\def\re{{\Bbb{R}}}
\def\bc{{\mathbb C }}
\def\fT{{\frak T}}
\def\nb{{\mathbb N}}
\def\bz{{\mathbb Z}}

\def\pc{{\mathcal P}}

\centerline{\huge\bf Entropy ratio for infinite sequences }

\medskip

\centerline{\huge\bf with positive entropy}

\vskip .3in

\centerline{\sc Christian Mauduit}

\centerline{\sl Universit\'e d'Aix-Marseille et Institut Universitaire de France}

\centerline{\sl Institut de Math\'ematiques de Marseille, UMR 7373 CNRS,}

\centerline{\sl 163, avenue de Luminy, 13288 Marseille Cedex 9, France}

\vskip .2in

\centerline{\sc Carlos Gustavo Moreira}

\centerline{\sl Instituto de Matem\'atica Pura e Aplicada}

\centerline{\sl Estrada Dona Castorina 110}

\centerline{\sl 22460-320 Rio de Janeiro, RJ, Brasil}

\vskip .3in

{\bf Abstract:}
The complexity function of an infinite word $w$ on a finite alphabet $A$ is the sequence counting, for each non-negative $n$, the number of words of length $n$ on the alphabet $A$ that are factors of the infinite word $w$.
For any given function $f$ with exponential growth, we introduced in [MM17] the notion of {\it word entropy} $E_W(f)$ associated to $f$ and we 
described the combinatorial structure of sets of infinite words with a complexity function bounded by $f$.
The goal of this work is to give estimates on the word entropy $E_W(f)$ in terms of the  limiting lower exponential growth rate of $f$.

\vskip .1in



This work was supported by CNPq, FAPERJ and the Agence Nationale de la Recherche project ANR-14-CE34-0009 MUDERA.
\vskip .3in

\section{Notations}

We denote
by $q$ a fixed integer greater or equal to $2$, by $A$ the finite
alphabet $A=\{0,1,\dots,q-1\}$, by $A^*=\bigcup\limits_{n\ge0} A^n$
the set of finite words on the alphabet $A$ and by $A^{\N}$ the set
of infinite words (or infinite sequences of letters) on the alphabet
$A$.
More generally, if $\Sigma \subset A^*$, we denote by  $\Sigma ^ \N$ the set of infinite words obtained by concatenating elements of $\Sigma$.
If $v\in A^n, n \in \N$ we denote $|v| = n$ the length of the word $v$ and if $S$ is a finite set, we denote by $|S|$ the number of elements of $S$.
If $w\in A^{\N}$ we denote by $L(w)$ the set of finite factors of $w$:
$$
L(w)=\{v\in A^*,\,\, \exists \, (v',v'')\in A^*\times A^{\N}, \, w=v'v v''\}
$$
and, for any non-negative integer $n$, we write $L_n(w)=L(w)\cap A^n$.
%For any $n \in \N$ we say that $v \in L_n(w)$ is a special factor (of length $n$) of $w$ if $v$ is prefix of more than one word in $L_{n+1}(w)$.
If  $x$ is a real number, we denote
$
\lfloor x\rfloor = \max\{n\in\Z, n\le x\},
\lceil x\rceil=\min\{n\in \Z, x\le n\}
$
and
$\{ x \} = x - \lfloor x\rfloor$.

Let us recall the following classical lemma concerning sub-additive sequences due to Fekete [Fek23]:

\begin{lemma}\label{lemFekete}
If $(a_n)_{n\ge1}$ is a sequence of real numbers such that $a_{n+n'} \le a_n + a_{n' }$ for any positive integers $n$ and $n'$, then the sequence
$\left( \frac {a_n} n\right)_{n\ge1}$ converges to $\inf_{n\ge 1}\frac {a_n} n$.
\end{lemma}

\begin{definition}\label{def1.2}
  The { \it complexity function} of $w\in A^{\N}$ is defined for any non-negative integer
  $n$ by $p_w(n)=|L_n(w)|$.

\end{definition}
%The complexity function gives information about the statistical
%properties of an infinite sequence of letters. In this sense, it
%constitutes one possible way to measure the random behavior of an
%infinite sequence (see [Que87, PF02]).
\noindent
For any $w\in A^{\N}$ and for any $(n,n')\in \N^2$ we have
  $L_{n+n'}(w)\subset L_n(w) L_{n'}(w)$ so that 
$$
  p_w(n+n')\le p_w(n) p_w(n')
$$
and it follows then from 
Lemmas \ref{lemFekete}  applied to $a_n = \log p_w(n)$
that for any
$w\in A^{\N}$, the sequence $\left( \frac 1n \log p_w(n)\right)_{n\ge1}$ converges
 to $\inf_{n\ge 1} \frac 1n \log p_w(n)$.

%\medskip

We denote
$$E(w)=\lim\limits_{n\to\infty}\frac 1n\log p_w(n) = h_{top} (X(w), T)$$
%can be shown (see for example [K\accent23ur03]) that $E(w)$ is equal to $h_{top} (X(w), T)$, 
the topological entropy of the symbolic dynamical system $(X(w),T)$ where
$T$ is the one-sided shift on $A^{\N}$ and $X=\overline{orb_T(w)}$ is
the closure of the orbit of $w$ under the action of $T$ in $A^{\N}$
(see for example [Fer99] or [PF02] for a detailed study of the notions of complexity function and topological entropy).
%($A^{\N}$ is equipped with the product topology of the discrete
%topology on $A$, i.e. the topology induced by the distance
%$d(w,w')= \exp(-\text{min} \{n \in \N | \, w_n \ne w_n' \})$).

If $f$ is a function from $\N$ to $\R^{+}$, we consider the set
$$W(f)=\{w\in A^{\N}, p_w(n)\le f(n), \forall n \in \N\}$$
and we denote
$$\sL_n(f)=\bigcup\limits_{w\in W(f)} L_n(w).$$ 
For any $(n,n')\in \N^2$ we have
$\sL_{n+n'}(f)\subset \sL_n(f) \sL_{n'}(f)$ so that we can deduce from Lemma \ref{lemFekete} applied to $a_n = \log |\sL_n(f)|$ that the sequence
$\left( \frac 1n \log |\sL_n(f)| \right)_{n\ge1}$ converges
 to \break $\inf_{n\ge 1} \frac 1n \log |\sL_n(f)|$,
 which is the topological entropy of the subshift $(W(f), T)$ :
 $$h_{top} (W(f), T) = \lim_{n\to+\infty}\frac 1n\log |\sL_n(f)| = \inf_{n\ge 1} \frac 1n \log |\sL_n(f)|.$$
We denote by $E_0(f)$ the limiting lower exponential growth rate of $f$
$$E_0(f)=\lim\limits_{n\to\infty} \inf \frac 1n \log f(n).$$

\section{Presentation of the results}

%\begin{definition}\label{def2.1}

%We say that a function $f$ from $\N$ to $\re^+$ verifies the conditions $(\mathcal C)$ if

%(i) the sequence $\left( f(n)\right)_{n\ge1}$ is strictly increasing;

%(ii) $\exists\,\, n_0\in\N$, $\forall n\ge n_0 \Rightarrow f(2n)\le
%f(n)^2 \text{ and }f(n+1) \le f(1) f(n) $;

%(iii) the sequence $\left( \frac 1n \log f(n)\right)_{n\ge1}$ converges.

%\end{definition}

Our work concerns the study of infinite sequences $w$ the complexity
function of which is bounded by a given function $f$ from $\N$ to $\R^{+}$.
We studied in [MM10] and [MM12] the case $E_0(f)=0$ and
we considered in [MM17] the case of positive entropy, for which few results were known since the work of Grillenberger [Gri].
We defined in [MM17] the notion of w-entropy (or word-entropy) of $f$ as follow :

\begin{definition}\label{defg}
If $f$ is a function from $\N$ to $\R^{+}$, the
{\it w-entropy} (or  {\it word entropy}) of $f$ is the quantity
$$E_W (f) =\sup_{\substack{w \in W(f)}} E(w).$$
\end{definition}

We gave in [MM17] a combinatorial proof of the fact that $E_W(f)$ is equal to the topological entropy of the subshift
$(W(f), T)$ (notice that this can be also obtained as a consequence of the variational principle)
and that it allows to compute exactly the fractal dimensions of the set
%\begin{equation} \label {C(f)}
%C(f)=\{x = \sum \limits_{i \ge 0} \frac{w_i}{q^{i+ 1}} \in [0,1] , w(x) = {w_0}{w_1}\cdots{w_i}\cdots \in W(f)\}
%\end {equation}
of real numbers from the interval $[0,1]$ the $q-$adic expansion of which has a
complexity function bounded by $f$.

This paper is devoted to
the study of the properties of the w-entropy $E_W$ and its relations with the  limiting lower exponential growth rate $E_0$.
Infinite words whose complexity function has an exponential growth but low initial values play a special role in this study and we define the following important class of infinite words
on the alphabet $\{0, 1 \}$ for which we provide in Section \ref{pre-sturmian} a useful renormalization theorem (Theorem \ref {lem_normalizable1}):

\begin{definition}\label{pre-sturmian}
We say that $w \in \{0, 1 \}^\N $ is a  {\it pre-sturmian} infinite word of order $k$ if $w$ is not ultimately periodic and if, for any non-negative integer $n \le k$, we have $p_w (n) = n + 1$.
\end{definition}

\begin{remark}
It follows from a classical result due to Coven and Hedlund ([CH73]) that non ultimately periodic infinite words $w$
with lowest possible complexity function $p_w$ are the ones for which $p_w(n)=n+1$ for any non-negative
integer $n$. Such infinite words, called sturmian words, have been
extensively studied since their introduction by Hedlund and
Morse in [HM38] and [HM40] (see [Lot02, chapter 2] and [PF02, Chapter 6] for very good surveys on sturmian words).
It follows from Definition \ref {pre-sturmian} that a  sturmian word is a pre-sturmian word of any order.
\end{remark}

\begin{definition}\label{defrho}
If $f$ is a function from $\N$ to $\R^{+}$, we call
{\it entropy ratio} of $f$ the quantity
$$\rho(f) = \frac {E_W (f)} {E_0 (f)}.$$
\end{definition}

It follows from the definitions of $E_0$ and $E_W$ that we always have $\rho (f)\le 1$  and it is easy to give examples of function $f$ for which 
$\rho (f)$ can be made arbitrarily small (see beginning of Section 7 from [MM17]).
On the other hand, if $f$ is indeed a complexity function (i.e. $f=p_w$ for
some $ w \in A^\N$), then we clearly have $\rho (f) = 1$.
But it seems difficult to
find a set of simple conditions on $f$ which hold for complexity functions and implies $\rho (f) = 1$
(see Problem 2.5 from [MM17]).

We will suppose in this paper that functions $f$ satisfy the following quite natural conditions $(\mathcal C^*)$
which hold for all unbounded complexity functions:
\begin{definition}\label{defC^*}
We say that a function $f$ from $\N$ to $\R^{+}$ satisfies the conditions $(\mathcal C^*)$ if

i) for any $n \in \N$ we have $f(n+1) > f(n) \ge n+1$ ;

ii) for any $(n, n' ) \in \N^2$ we have $f(n+n') \le f(n) f(n') $.
\end{definition}

\begin{remark}
If there exists $n \in \N$ such that $f(n) \le n$, then any $w \in A^{\N}$ such that $p_w \le f$ is ultimately periodic, so that $W(f)$ is finite.
\end{remark}

\begin{remark}\label{rmk5.4}
Given any function $f$ from $\N$ to $\R^{+}$ such that $f(n)  \ge n+1$ for any $n \in \N$,
it is possible to construct recursively a non increasing integer valued function $\tilde f$ satisfying  the condition $(\mathcal C^* (ii))$
and a real valued function 
$\tilde { \tilde f}$ satisfying  the conditions $(\mathcal C^*)$ such that
$W(f) = W(\tilde f) = W(\tilde { \tilde f})$ and such that $E_0(f) >0$ implies $E_0(\tilde f) = E_0(\tilde {\tilde f}) >0$
(see Remark 7.3 from [MM17]).
%For example, we can take $\tilde f$ and $\tilde { \tilde f}$ defined by $\tilde f (0) = 1$, $\tilde f(1) = \lfloor f(1) \rfloor$ and, for any integer $n \ge 2$,
%$$\tilde f(n) = \min \{\inf_{\substack{ k \ge n}} \lfloor f(k) \rfloor, \min_{\substack{ 1 \le k < n}}  \tilde f (k) \tilde f (n-k) \}$$ and, for any $n \in \N$,
%$$\tilde { \tilde f} (n) = \tilde f (n) + \frac n {n+1}.$$
\end{remark}

\begin{remark}\label{rmkC^*}
If a function $f$ from $\N$ to $\R^{+}$ satisfies the
conditions $(\mathcal C^*)$ then, for any $ n \in \N$, we have $f(n) \ge \max \{n+1, \exp (E_0(f) n)  \}$ (see Lemma 7.4 from [MM17]).
\end{remark}

%\begin{proof}

%This is a consequence of Lemma \ref {lemFekete} applied to $a_n = \log f(n)$.

%\end{proof}

In [MM17] we showed that, even when the function $f$ satisfies the conditions $(\mathcal C^*)$, it might happen that $\rho (f) < 1$.
The main goal of this paper is to answer Problem 6.5 from [MM17] by showing

\begin{theorem} \label {theorem1/2}
We have $$\inf \{ \rho (f), f \mbox{  $satisfies$  } (\mathcal C^*) \} = \frac 12.$$
\end{theorem}

The proof of Theorem \ref {theorem1/2} will follow from Section \ref{>1/2} where we prove that if
$f$ satisfies the conditions $(\mathcal C^*)$, we always have
$\rho (f) > \frac 12$ (Theorem \ref{1/2theorem}) and from Section \ref{=1/2} where we prove that $\frac 12$ is the best possible constant (Theorem \ref{1/2theoremoptimal}).
The proof of Theorem \ref{1/2theoremoptimal} is based on the renormalization theorem concerning pre-sturmian infinite words given in Section \ref{pre-sturmian} (Theorem \ref{lem_normalizable1}).

%To each $w \in A^\N$ we can also associate the real number in [0,1]
%whose continued fraction expansion is coded by $w$. More precisely,
%the function $\rho': A^{\N} \longrightarrow [0,1]$ given by
%\begin{center}
 % $\rho'(w)= [0; w_0+1, w_1+1,\cdots, w_n+1,\cdots]$ if $w =
 % {w_0}{w_1}\cdots{w_n}\cdots$ with $w_{n} \in A$ for any $n \ge 0$,
%\end{center}
%allows us to translate results concerning a given $L \subset A^\N$ to
%results concerning $\rho'(L) \subset [0,1]$.

%In this case, $\rho'(A^\N)$ is a small subset (a Cantor set of Hausdorff
%dimension smaller than one) of [0,1], consisting of the real numbers whose continued fraction expansion has all coefficients bounded by $q$.

%In section 5.5, we give results similar to the one obtained for $C(f)$ in
%the case of the set $C'(f)=\rho'(W(f))$ of the real numbers whose
%continued fraction expansion is bounded by $q$ and has a complexity
%bounded by $f$.

%\medskip

\section{Renormalization of pre-sturmian infinite words of order $k$} \label{pre-sturmian}

\begin{lemma}\label{lem_normalizable2}

If $a$ and $b$ are words on the alphabet $\{0,1\}$ with distinct first letters and $|a|\ge |b|$, $s$ is a positive integer and $w$ is a non ultimately periodic word that belongs to $\{a, ba^s\}^\N$ such that
$p_w(k)=k+1$ and $(s+1)|a|+|b| \le k$ then, if we consider $w$ as an infinite word on the two letters alphabet $\{A, B\}$ with $A=a$ and $B=ba^s$,
the words $AA$ and $BB$ cannot be both factors of $w$ (except in the first positions). More precisely, if $w=\gamma_1 \gamma_2 \dots$ with $\gamma_k\in \{A, B\}, \forall k\ge 1$, then there are no indexes $i, j$ with $1<i<j$, $\gamma_i=\gamma_{i+1}=B$ and $\gamma_j=\gamma_{j+1}=A$.

\end{lemma}

\begin{proof}
If we put $r=(s+1)|a|+|b| \le k$, we have $p_w(r)=r+1$, so that $w$ has only one special factor of length $r-1$.
Let us show that the words $aba^s$ and $baa^s=ba^{s+1}$ (of size $r$) cannot be both special factors of $w$.
If the words $aba^s$ and $ba^{s+1}$ were both special factors, then (we denote by $T$ the one-sided shift map which deletes the first letter of a sequence) 
the words $T(a)ba^s$ and $T(b)a^{s+1}$ (of size $r-1$) would also be special factors of $w$, and thus $T(a)ba^s=T(b)a^{s+1}$,
which would imply $T(ab)=T(a)b=T(b)a=T(ba)$. This would be a contradiction, since the first letters of $a$ and $b$ are different, and the total number of 0's 
(and 1's) in $ab$ is equal to that in $ba$, so the total number of 0's in $T(ab)$ is different from that in $T(ba)$.

Let us remark that a factor $ba^s$ that is not in the beginning of $w$ is necessarily preceded by a factor $a$ (since $ba^s$ ends by $a$). Consider now the two possible cases:

1) if $aba^s$ is not a special factor, then\\
- if $aba^s$ cannot be followed by $ba^s$, the factor $ba^sba^s$ cannot appear (except in the first positions), which implies that $BB$ is not a factor of $w \in \{A, B\}^\N$;\\
- if $aba^s$ cannot be followed by $a$, it is necessarily followed by $ba^s$, so that $w$ would end by $ba^sba^sba^s\dots$, which would contradict the non ultimate periodicity of $w$.

2) if $aba^s$ is a special factor, then $ba^sa=ba^{s+1}$ is not a special factor, and thus\\
- if $ba^sa=ba^{s+1}$ cannot be followed by $a$, some iterate of $w$ under the shift $T$ (after the first $ba$) cannot have two consecutive factors $a$,
which implies that $AA$ is not a factor of $w \in \{A, B\}^\N$ (except in the first positions);\\
- if $ba^sa=ba^{s+1}$ cannot be followed by $ba^s$, it is necessarily followed by $a$. Let $u$ be the word formed by the $|ba^s|=s|a|+|b|$ last letters of $a^{s+1}$.
If $u \neq ba^s$, since $ba^s$ is a special factor, $u$ cannot be a special factor, so $w$ would be necessarily followed by $a$, and thus $w$, after the first occurrence of $ba^sa=ba^{s+1}$, would end by $aaaa\dots$, which would contradict the non ultimate periodicity of $w$.
Therefore we should have $u=ba^s$, but this would imply that $ba^{s+2}=ba^{s+1}a$ ends by $ca=ba^sa=ba^{s+1}$, which is necessarily followed by $a$, and thus, again, $w$ would end by $aaaa\dots$ after the first occurrence of $ba^sa=ba^{s+1}$, which would also contradict the non ultimate periodicity of $w$.
\end{proof}

\begin{theorem}\label{lem_normalizable1}
Any pre-sturmian infinite word $w \in \{0,1\}^\N$ of order $k$ is renormalizable in the following way:
there are two words $a$ and $b$ belonging to $\{0,1\}^*$ with distinct first letters and a positive integer $s$ such that
$$|a| \ge |b|, (s+1)|a|+|b| > k$$
and some iterate of $w$ under the one-sided shift $T$ belongs to $\{a, ba^s\}^\N$.
\end{theorem}

\begin{proof}
We first notice that, since $p_w(2)=3$, the words $00$ and $11$ cannot be both factors of $w$ (otherwise $w$ would be ultimately periodic).
This implies that $w$ is renormalizable with words $0$ and $10$ (if $w$ does not contain the factor $11$) or with words $1$ and $01$ (if $w$ does not contain the factor $00$).

Assume now that $w$ is renormalizable with words $a$ and $ba^s$ with distinct first letters and $(s+1)|a|+|b|$ maximal. If we had $(s+1)|a|+|b| \le k,$ it would follow from Lemma \ref{lem_normalizable2} that, as before, some iterate of $w$ would be renormalizable with words $a$ and $ba^sa=ba^{s+1}$ or with words $ba^s$ and $aba^s=a(ba^s)^1$, which would contradict in both cases the maximality of $(s+1)|a|+|b|$.
\end{proof}

\section{The entropy ratio is always bigger than $\frac 12$} \label {>1/2}
% $\frac {E_W(f)} {E_0(f)}$
We begin by proving the following elementary Lemma :

\begin{lemma}\label{lem_beta}
If $k$ is a fixed positive integer and $\beta > 1$ a solution of $\beta ^{k+1} = \beta ^{k} +1$, then

i) for any $r \in \N$, we have $\beta ^{k+r} > r + 1$ ;

ii) for any $r \in \N$, we have $k + 1 +r(r+3)/2 < \beta ^{2(k+r)}.$
\end{lemma}

\begin{proof}

We have $\beta ^{k} > 1$ and, if we suppose that $\beta ^{k+r} > r + 1$, it follows that $\beta ^{k+r+1} = \beta ^{k+r} + \beta ^{r} > (r + 1) + 1$, which proves
i) by induction on $r$.

Then we can easily prove ii) by induction on $r$.
When $r=0$, the assertion follows from i) applied to $r=k$ and, if we suppose that $k + 1 +r(r+3)/2 < \beta ^{2(k+r)}$, it follows that
$k + 1 + (r+1)(r+4)/2 = k + 1 +r(r+3)/2 + r +2 < \beta ^{2(k+r)} +r+2$.
By i) we have $\beta ^{2(k+r)} +r+2 < \beta ^{2(k+r)} + \beta ^{k+r+1}$, so that it remains to show that $\beta ^{2(k+r)} + \beta ^{k+r+1} < \beta ^{2(k+r+1)}$, i. e. that
$\beta ^{k+r} + \beta  < \beta ^{k+r+2} = \beta ^{k+r+1} + \beta ^{r+1} =  \beta ^{k+r} + \beta ^{r} + \beta ^{r+1}$,
which follows from $\beta < 1 + \beta$ if $r=0$ and $1 < \beta ^{r-1} + \beta ^{r}$ if $r \ge 1$.

\end{proof}

\begin{theorem} \label{1/2theorem}
%$\,$ \hfill \break
If $f$ is a function from $\N$ to $\R^{+}$ such that $f(n) \ge \max \{n+1, \exp (E_0(f) n)  \}$ for any $ n \in \N$, then $E_W(f)> \frac 12 E_0(f)$.
\end{theorem}
\begin{remark}
In particular, it follows from Remark \ref {rmkC^*} that Theorem \ref {1/2theorem} applies when $f$ 
satisfies the conditions $(\mathcal C^*)$.
\end{remark}

\begin{proof}
We split the proof in three parts depending on the value of $E_0(f)$.

If $E_0(f) \ge \log 2$ and $m = \lfloor \exp(E_0(f)) \rfloor \ge 2$, it follows from our hypothesis that $f(n) \ge m^n$ for any $n \in \N$.
If $w$ is a normal sequence on the alphabet $\{1,2,\dots,m \}$, we have $p_w(n)=m^n$ for any $n \in \N$, so that $w \in W(f)$ and thus
$$E_W(f) \ge E(w) =  \lim_{n\to+\infty}\frac 1n\log p_w(n)=\log m > \frac{\log(m+1)}2 >  \frac{E_0(f)}2.$$

If $\frac 12 \log 3 \le E_0(f) <\log 2$, it follows from  our hypothesis that $f(n) \ge \max \{n+1, 3^{n/2}\} \ge F_{n+2}$ for any $n \in \N$,
where $(F_n)_{n\in\N}$ is the Fibonacci sequence, defined in Section \ref {5.3}.
For almost any infinite words $w$ on the alphabet $A=\{0,1\}$ without the factor $11$ we have
$p_w(n) = F_{n+2}$ for any $n \in \N$, so that we have
$$E_W(f) \ge E(w) = \lim_{n\to+\infty}\frac 1n\log p_w(n)=\log (\frac{1+\sqrt 5}{2}) >  \frac 12 \log 2 > \frac{E_0(f)}2.$$  

If $E_0(f) < \frac 12 \log 3 $, let us consider,
%the family of shifts $\Sigma_k$ over the alphabet $\{0,1\}$,
for any positive integer $k$,
%and $\Sigma_k$ is the complete shift generated by $0$ and $1 0^k$
the set $\Sigma_k = \{0,10^k\}^\N$
(i. e. the set of infinite words over the alphabet $\{0,1\}$ such that two occurrencies of the letter $1$ are always separated by at least $k$ occurrencies
of the letter $0$).

For any $n \in \N$, let $L_k(n)$ be the set of the words of length $n$ over the alphabet $\{0,1\}$ such that two occurrencies of the letter $1$
are always separated by at least $k$ occurrencies
of the letter $0$, and let $q_k(n)=|L_k(n)|$.
For almost any $w \in \Sigma_k$ we have $p_w(n)=q_k(n)$ for any $n \in \N$
and it is not difficult to construct an infinite word $w=w^{(k)} \in \Sigma_k$ which satisfy these equalities: if we enumerate, for each
$n \ge 1$, $L_k(n)=\{\al_1^{(n)},\al_2^{(n)},\dots,\al_{q_k(n)}^{(n)}\}$, and take  
$$w^{(k)}=\al_1^{(1)}0^k\al_2^{(1)}0^k\dots\al_{q_k(1)}^{(1)}0^k\al_1^{(2)}0^k\al_2^{(2)}0^k\dots0^k\al_{q_k(2)}^{(2)}0^k\dots\al_1^{(n)}0^k\al_2^{(n)}0^k\dots0^k\al_{q_k(n)}^{(n)}0^k\dots.$$
then we have $p_w^{(k)}(n)=q_k(n)$ for any $n \in \N$.

For any $n \in \N$  we have
$$L_{k}(n+k+1) = L_{k}(n+k) 0 \cup L_{k}(n+k) 1 = L_{k}(n+k) 0 \cup L_{k}(n) 0^k 1,$$
which implies that, for any fixed positive integer $k$, the sequence $(q_k(n))_{n \in \N}$ satisfies the 
following recurrence, valid for any $n \in \N$: 
$$q_k(n+k+1)=q_k(n+k)+q_k(n).$$ 
Moreover the sequence $(q_k(n))_{n \in \N}$ satisfies $q_k(n)=n+1$ for $0 \le n \le k+1$
(a word in $L_k(n)$ for $0 \le n \le k+1$ has at most one letter equal to $1$).
This implies in particular that, for $0 \le r \le k+1$, we have
%\begin{equation} \label {q_k}
$$q_k(k+r+1)=k+2+r(r+3)/2.$$
%\end{equation}
This also implies that there are two positive constants $c_k$ and $d_k$ such that for any $n \in \N$ we have $c_k.\be_k^n<q_k(n)=p_{w^{(k)}}(n)<d_k.\be_k^n$,
where $\be_k$ is the largest real root of the polynomial $x^{k+1}-x^k-1$. 

%If we write $\be_k=\exp(\ve)$, then we have $1=\be_k^k(\be_k-1)$, so $1$ is asymptotically $\exp(k \ve) \ve=\exp(k \ve +\log \ve)$, so $\log \ve$ is asymptotically $-k \ve$, which implies that $\ve$ is asymptotically $\frac{\log k}k$, and so $\be_k$ behaves like $\exp(\frac{\log k}k)$ or, equivalently, to $1+\frac{\log k}k$ for $k$ large. 

%On the other hand, we have $\max\{\frac{\log q_k(n)}{n}, n>k\}=\max \{\frac{\log q_k(k+r+1)}{k+r+1}, 0 \le r \le k+1\}$.

Let us denote $$\gamma_k=\max_{\substack{k+1 \le n \le 2(k+1)}} {\frac{\log q_k(n)}{n}} = \max_{\substack{0 \le r \le k+1}} {\frac{\log (k+2+r(r+3)/2)}{k+r+1}}$$
and remark that $\gamma_k=\max_{\substack{ n \ge k+1}} {\frac{\log q_k(n)}{n}}$.
Indeed, if we write any integer $n \ge k+1$ as $n = b(k+1) + k+ r+1$, with $b \ge 0$ and $0 \le r \le k+1$, we get
$q_k(n) = q_k(b(k+1) + k+r+1) \le (q_k(k+1))^b q_k(k+r+1) \le \exp(b (k+1) \gamma_k ) \exp((k+r+1) \gamma_k ) =  \exp(\gamma_k n).$

It follows that for any integer $k \ge 1$, we have $ \log \be_k \le \gamma_{k}$ and
it follows from Lemma \ref {lem_beta} ii) applied to $\beta = \beta_k$ that, for any integer $k \ge 2$, we have
$$\gamma_{k-1} = \max_{\substack{0 \le r \le k}} {\frac{\log (k+1+r(r+3)/2)}{k+r}} < 2 \log \be_k.$$
%(indeed, this is equivalent to $k+1+r(r+3)/2)<\be_k^{2(k+r)}$ for any $0 \le r \le k$).

 This is enough to finish the proof. Indeed, for any integer $n \ge k+1$, we have $L_k(n) \subset L_{k-1} (n)$, so that $q_k(n) < q_{k-1} (n)$ and
$$\gamma_k=\max_{\substack{ n \ge k+1}} {\frac{\log q_k(n)}{n}} < \max_{\substack{ n \ge k}} {\frac{\log q_{k-1}(n)}{n}} = \gamma_{k-1}.$$
Since  $$\gamma_k= \max_{\substack{0 \le r \le k+1}} {\frac{\log (k+2+r(r+3)/2)}{k+r+1}}<\frac {\log (k+2+(k+1)(k+4)/2)}{k+1},$$
we have $\lim_{n\to+\infty}\gamma_k=0$.
Now, as $E_0(f) < \frac 12 \log 3$ and $\gamma_1 =  \frac 12 \log 3$, let $k_0 \ge 2$ be such that $\gamma_{k_0} \le E_0(f) < \gamma_{k_0-1}$.
For any $n \in \N$, we have $$f(n) \ge \max \{n+1, \exp (E_0(f) n)  \} \ge \max \{n+1, \exp (\gamma_{k_0}  n)  \} \ge q_{k_0}(n),$$ so that
$w^{(k_0)} \in W(f)$, which implies $$E_W(f) \ge \lim_{n\to+\infty}\frac{\log q_{k_0}(n)}{n}=\log \be_{k _0}> \frac 12 \gamma_{k_0-1} > E_0(f)/2.$$

\end{proof}

\section{The constant $\frac 12$ is optimal} \label {=1/2}

The constant $\frac 12$ in Theorem \ref {1/2theorem} cannot be improved, as shows the following theorem:

\begin{theorem} \label{1/2theoremoptimal}
%$\,$ \hfill \break
For any $c> \frac 12$, there is a function $f$ from $\N$ to $\R^{+}$ satisfying the conditions $(\mathcal C^*)$ such that $\rho(f)<c$.
\end{theorem}
\begin{remark}
By the proof of Theorem \ref {1/2theorem} it follows that, if $c$ is very close to $\frac 12$ in Theorem \ref {1/2theoremoptimal}, then $E_0(f)$ is necessarily very close to $0$.
\end{remark}
\begin{proof}
Let $k$ be a large positive integer, $\theta = \exp (\frac {\log k} {k})$ and define the function $f$ by $$f(n)=\max\{n+1,\theta^{n}\}.$$
We have $E_0(f)= \log \theta = \frac {\log k} {k}$ and our goal is to give, for any $w \in W(f)$, an upper bound for $E(w)=\lim\limits_{n\to\infty}\frac 1n\log p_w(n)$.
As ultimately periodic infinite words have zero entropy, we can restrict ourselves to non ultimately periodic infinite word $w \in W(f)$.
Such infinite words are pre-sturmian of order $k$ (we have $p_w (k) \le f(k) = k+1$), so that if follows from Theorem \ref {lem_normalizable1}
that $w$ is renormalizable with words $A = a$ and $B = ba^s$ with different first letters, $|a| \ge |b|$ and $(s+1)|a|+|b|>k$.
Any factor of length $n$ of $w \in \{0,1\}^\N$ can be covered by a factor of length at most  $\lceil \frac n {|a|} \rceil + 1$ in the representation of $w$ as an infinite word on the
alphabet $\{A, B\}$.
As there are at most $2^{|a|+|b|-1}$ possible choices for the words $A$ and $B$, it follows that $p_w (n) \le 2^{|a|+|b| + \lceil \frac n {|a|} \rceil}$ and
$$E (w) = \lim_{n\to+\infty}\frac{\log p_w(n)}n \le \frac1 {|a|} \log 2.$$
It follows in particular that, if $E(w) \ge \frac 1 2 \log \theta = \frac{\log k}{2k}$, then $|a| <\frac{2k \log 2}{\log k}$ and, since $(s+2)|a| \ge (s+1)|a|+|b| > k$,
$s$ should be large for $k$ large.
%We will now concentrate in estimating the entropy of a word $w$ as before, i.e., renormalizable with words $a$ and $ba^s$ with $s$ large, different initial elements and with $(s+1)|a|+|b|>k$,

The infinite word $w$ can be written as
\begin{equation} \label {w}
w= a^{s_0} b a^{s_1} \dots b a^{s_j} \dots,
\end{equation}
with $s_0 \ge 0$ and $s_j  \ge s$ for any integer $j \ge 1$.
We will concentrate on the gaps of size less than $2k$ between two consecutive occurrencies of the first letter of the word $b$ in the representation (\ref {w}).
Let $\ve$ be a small positive constant such that 
\begin{equation} \label {epsilon}
\frac{2 \log \log k}{\log k} \le \ve  \le \frac14
\end{equation}
to be chosen later and denote
$$\{r_1,r_2,\dots, r_p\} = \{s_j |a|+|b|, j \in \N\} \cap \{(1 - \ve)k, \dots, 2k-1\}.$$
For each $0 \le r \le \lceil 1/\ve \rceil$,
let $X (r)$ be the set of the indices $j$ with $$(1+(r-1) \ve)k\le r_j < (1+r \ve)k.$$
If we consider $h=\lfloor  \frac {\ve k} {|a|} \rfloor \le  \lfloor\ve (s+2) \rfloor$,
it follows from (\ref {epsilon}) that $h \le \lfloor (s+2)/4 \rfloor \le s/2$ so that, for any $j \in X(r)$, the word
$$e_1^{(j)}e_2^{(j)}\dots e_q^{(j)}=a^{2h}ba^{(r_j-|b|)/|a|}ba^{2h}$$ is a factor of $w$.
If $t_r=\lfloor (1+(r+2) \ve)k \rfloor$,
we claim that the $\lceil \ve k \rceil |X(r)| \le |a|(h+1) |X(r)|$ factors of size $t_r$ of $w$
$$e_i^{(j)}e_{i+1}^{(j)}\dots e_{t_r+i-1}^{(j)}, 1 \le i \le \lceil \ve k \rceil, j \in X(r)$$
are distinct. Indeed

- if $j_1<j_2$ and $e_i^{(j_1)}e_{i+1}^{(j_1)}\dots e_{t_r+i-1}^{(j_1)} = e_{i}^{(j_2)}e_{i+1}^{(j_2)}\dots e_{t_r+i-1}^{(j_2)}$,
we would have $e_{(2h-1)|a|+r_{j_1}+1}^{(j_1)}$ equal to the first letter of $b$, while $e_{(2h-1)|a|+r_{j_1}+1}^{(j_2)}$ is equal to the first letter of $a$, which would be a contradiction.

- if $i<\ell$ and $e_i^{(j_1)}e_{i+1}^{(j_1)}\dots e_{t_r+i-1}^{(j_1)} = e_{\ell}^{(j_2)}e_{\ell+1}^{(j_2)}\dots e_{t_r+\ell-1}^{(j_2)}$,
we would have $e_{n-\ell+i}^{(j_1)}=e_n^{(j_2)}$ for $\ell \le n \le t_r+\ell-1$ and, since
$e_m^{(j_l)}=e_{m+|a|}^{(j_l)}, l=1,2$ while $m+|a|<2|a|h$, we would have $e_{2|a|h+1}^{(j_2)}=e_{2|a|h+1-\ell+i}^{(j_1)}=e_{2|a|h+1-\ell+i-|a|}^{(j_1)}=e_{2|a|h+1-|a|}^{(j_2)}$,
which would lead to a contradiction, since $e_{2|a|h+1}^{(j_2)}$ is the first letter of $b$, and $e_{2|a|h+1-u}^{(j_2)}$ is the first letter of $a$, and they are distinct.\\
It follows that $$\ve k |X(r)| \le \lceil \ve k \rceil |X(r)| \le p_w(t_r) \le \theta^{t_r} \le k^{1+(r+2)\ve},$$ so that
$|X(r)| \le \frac1{\ve} k^{(r+2)\ve}$.

For $a$ and $b$ fixed, let us consider now, for any $n \in \N$, the number $y_n$ of factors of length $n$ of $w$ of the form
$v_1v_2\dots v_m$ with, for any  $ j \in \{1, \dots , m\}$, $v_j \in \{a,b\}$ and such that $v_m=b$.
For $k$ large enough, it follows from (\ref{epsilon}) that $\ve  \ge \frac{2 \log \log k}{\log k} \ge \frac{2 \log 2}{\log k}$, so that
$$(1-\ve)k  \le (1 - \frac{2 \log 2}{\log k})k \le k - |a| <  s |a|+|b|$$
and $$\inf_{j\ge 1} |ba^{s_j}| \ge (1-\ve)k.$$
It follows that for any integer $n \ge 2k$, we have 
$$y_n\le \sum_{j=1}^p y_{n-r_j}+\sum_{j=2k}^n y_{n-j},$$
so that $$\limsup_{n\to+\infty}\frac 1n\log y_n \le \log \hat \lambda,$$
where $\hat \lambda>1$ satisfies the equation
$$1=\sum_{j=1}^p {\hat \lambda}^{-r_j}+\sum_{j=2k}^\infty {\hat \lambda}^{-j}=\sum_{j=1}^p {\hat \lambda}^{-r_j}+\frac{{\hat \lambda}^{-2k}}{1-{\hat \lambda}^{-1}},$$
which implies
$$1 \le \sum_{r=1}^{\lceil 1/\ve \rceil} |X(r)| {\hat \lambda}^{-(1+(r-1) \ve)k}+\frac{{\hat \lambda}^{-2k}}{1-{\hat \lambda}^{-1}} \le \sum_{r=1}^{\lceil 1/\ve \rceil} \frac1{\ve} k^{(r+2)\ve}{\hat \lambda}^{-(1+(r-1) \ve)k}+ \frac{{\hat \lambda}^{-2k}}{1-{\hat \lambda}^{-1}}.$$
It follows from  (\ref{epsilon}) that
%$$\frac 1{\ve}  \le \frac{\log k}{2 \log \log k} = \exp (\log \log k - \log \log \log k - \log 2) \le \exp (\log \log k) = k^{\frac{\log \log k}{\log k}} \le k^{\ve /2}$$
$$\frac 1{\ve}  \le \frac{\log k}{2 \log \log k} = e^{(\log \log k - \log \log \log k - \log 2)} \le e^{(\log \log k)} = k^{\frac{\log \log k}{\log k}} \le k^{\ve /2}$$
so that, writing $\hat \lambda = e^{\sigma \log(k)/k}$ with $\sigma > 0$, we have
$$1 \le \sum_{r=1}^{\lceil 1/\ve \rceil} \frac1{\ve} k^{(r+2)\ve}k^{-(1+(r-1) \ve)\sigma}+ \frac{k^{-2\sigma}}{1-e^{-\sigma \log(k)/k}}=O(\frac 1{\ve}k^{-\sigma+\ve(2+\sigma)}k^{\ve\lceil 1/\ve \rceil(1-\sigma)}+ \frac {k^{1-2\sigma}} {\log(k)})$$
$=O(\frac1{\ve}k^{1-2\sigma+3\ve}) =O(k^{1-2\sigma+\frac 52\ve}),$
which implies that $\sigma < 1/2+2 \ve$, for $k$ large.

If a factor of length $n$ of $w$ contains at least one occurrence of the word $b$, it is of the form
$\beta_1 v_1v_2\dots v_m \beta_2$ where $\beta_1$ is a strict suffix of $a$ or $b$, $v_j \in \{a,b\}$ for any  $ j \in \{1, \dots , m\}$, $v_m=b$ and $\beta_2$ a prefix of $a^i$ for some $i \ge 0$.
As there are $n - |b| + 1$ possible choices for the emplacement of the first letter of the word $v_m$, there are at most
$$2(n - |b| + 1) \sum_{n' \le n} y_{n'}$$ such factors of length $n$.
On the other hand, there are at most $2|a|^2$ factors of length $n$ of $w$ without any occurrence of the word $b$ (there are the words of the the kind
$\beta_1 a^m \beta_2$, where $\beta_1$ is a strict suffix of $a$ or $b$ and $\beta_2$ a strict prefix of $a$ or $b$).
Finally, as there are at most $2^{|a| + |b| -1}$ possible words $a$ and $b$, we have
$$p_w(n) \le 2^{|a| + |b|} ((n - |b| + 1) \sum_{n' \le n} y_{n'} + |a|^2)$$
and it follows that
$$E(w) = \lim_{n\to+\infty}\frac 1n\log p_n(w) = \limsup_{n\to+\infty}\frac 1n\log y_n \le \log \hat \lambda = \sigma \frac {\log k}k < (1/2+2 \ve) \frac {\log k}k .$$
If we choose now $\ve=\frac12(c-1/2)$, where $c>1/2$ is given in the statement, we have
$$E_W (f) =\sup_{\substack{w \in W(f)}} E(w)  < c \frac {\log k}k ,$$
which concludes the proof.
\end{proof}

%given that $p_w(n) \le \max\{n+1,e^{n \log(k)/k} \}, \forall n \ge 1$ (in fact we will use this inequality only for $k<n\le 2k$). 

%$\be_j \in \{a,b\}, \forall j \ge 1$, between two words $\be_i, \be_j$ equal to $b$, there are at least $s$ words equal to $a$. Let $r_1,r_2,\dots, r_p$ the set $\{k<n<2k| n=(j-i-1)|a|+|b|; \exists 1 \le i<j, \be_j=\be_i=b, \be_t=a, i<t<j\}$.

%let $N(r)$ be the number of indices $j$ such that $(1+(r-1) \ve)k\le r_j < (1+r \ve)k$. 

%Despite the previous result, given any $h>0$ there is a word $w$ with entropy $h$ whose complexity function is close to $\max\{n+1,e^{hn}\}$:

\section*{Acknowledgement}
We thank the referee for his comments and suggestions.


\begin{thebibliography}{u{standard}}


%\bibitem[All94]{All94} J.-P. Allouche, Sur la complexit\'e des suites
 % infinies, \textit{Bull. Belg. Math. Soc. Simon Stevin}, 1(2):
 % 133-143, 1994.

%\bibitem[Bes34]{Bes34} A.S. Besicovitch, Sets of fractional dimensions:
%On rational approximation to real numbers,
%\textit{J. London Math. Society}, 9 (1934), 126-131.

%\bibitem[Bor09]{Bor09} E. Borel, Les probabilit\'es d\'enombrables et leurs applications arithm\'etiques,
%Rend. Circ. Mat. Palermo, 27 (1909), pp. 247�271.

%\bibitem[Bug04]{Bug04} Y. Bugeaud, Approximation by algebraic numbers,
%  {\it Cambridge Tracts in Mathematics}, 160, Cambridge University
%  Press, 2004.


\bibitem[CH73]{C-H} E.M. Coven and G.A. Hedlund, Sequences with minimal
  block growth, \textit{Math. Systems Theory} 7 (1973), 138-153.
  
%\bibitem[CK]{CK97} J. Cassaigne and J. Karhumaki, Toeplitz words,
% generalized periodicity and periodically iterated morphisms,
% \textit{Eur. J. Comb.}, 18(5):497-510, 1997.

%\bibitem[Egg49]{Egg49} H.G. Eggleston, The fractional dimension of a set
%defined by decimal properties, \textit{Quart. J. Math. Oxford Ser.}, 20 (1949), 31-36.


%\bibitem[Fal90]{Fal} K. J. Falconer, Fractal geometry. mathematical foundations and applications
 %John Wiley \& Sons, Chichester 1990.

\bibitem[Fek23]{Fek} M. Fekete, Uber der Verteilung der Wurzeln bei gewissen algebraischen Gleichungen mit ganzzahligen Koeffizienten, \textit{Mathematische Zeitschrift} 17 (1923), 228-249.

\bibitem[Fer99]{Fer99} S. Ferenczi, Complexity of sequences and
 dynamical systems, \textit{Discrete Math.}, 206(1-3):145-154, 1999.
  
%\bibitem[Goo41]{Goo41} I. J. Good, The fractional dimensional theory of
%continued fractions, \textit{Proc. Cambridge Philos. Soc.}, 37 (1941), 199-228.

\bibitem[Gri73]{Gri} C. Grillenberger, Construction of strictly ergodic
 systems I. Given entropy, \textit{Z. Wahrscheinlichkeitstheorie
  verw. Geb.} 25 (1973), 323-334.

\bibitem[HM38]{HM38} G. A. Hedlund and M. Morse, Symbolics dynamics,
  \textit{Amer. J. Math.} 60 (1938), 815-866.
  
\bibitem[HM40]{HM40} G. A. Hedlund and M. Morse, Symbolics dynamics
  II. Sturmian trajectories, \textit{Amer. J. Math.} 62 (1940), 1-42.
  
%\bibitem[LM06]{LM06} R. Labarca and C. G. Moreira, Essential dynamics for Lorenz maps on the real line and the Lexicographical World, \textit{Ann. I. H. Poincar� - Anal. Non Lin.} 23 (2006) 683-694.


%\bibitem[K\accent23ur03]{Kur03} P. K\accent23urka, Topological and symbolics dynamics,
 % \textit{Cours sp\'ecialis\'es SMF},volume 11, 2003.

\bibitem[Lot02]{Lot02} M. Lothaire, Algebraic combinatorics on words,
\textit{Encyclopedia of Mathematics and Its
 applications} 90, Cambridge University Press, 2002.

%\bibitem[Mat95]{Mat} P. Mattila, Geometry of sets and measures in
 % euclidean spaces, Cambridge University Press, 1995.
  
\bibitem[MM10]{MM10} C. Mauduit and C. G. Moreira, Complexity of infinite sequences with zero entropy,
\textit{Acta Arithmetica} 142 (2010), 331-346.

\bibitem[MM12]{MM12} C. Mauduit and C. G. Moreira, Generalized Hausdorff dimensions of sets of real numbers with zero entropy 
	expansion,
\textit{Ergodic Theory and Dynamical Systems} 32 (2012), 1073-1089.

\bibitem[MM17]{MM17} C. Mauduit and C. G. Moreira, Complexity and fractal dimensions for infinite sequences with positive entropy,
\textit{preprint}, https://arxiv.org/abs/1702.07698.
    
%\bibitem[MS97]{MS97} C. Mauduit and A. S\'ark\"ozy, On finite binary pseudorandom sequences.
%I. Measure of pseudorandomness, the Legendre symbol. \textit {Acta Arithmetica} 82
%(1997), 365-377.

%\bibitem[MS98]{MS98} C. Mauduit and A. S\'ark\"ozy, On finite binary pseudorandom sequences.
%II. The Champernowne, Rudin-Shapiro and Thue-Morse sequence, a further construction.
%\textit {J. Number Theory} 73 (2) (1998), 256-276.

\bibitem[PF02]{PF02} N. Pytheas Fogg. Substitutions in dynamics,
  arithmetics and combinatorics, \textit{Lecture Notes in Mathematics
    1794}, Springer, 2002. Edited by V. Berth\'e, S. Ferenczi, C.
  Mauduit and A. Siegel.

%\bibitem[PT92]{PT92} J. Palis and F. Takens,
%Hyperbolicity and sensitive chaotic dynamics at homolinic bifurcations :
%fractal dimensions  and infinitely many attractors,
%Cambridge University Press,1992. 

%\bibitem[Que87]{Que87} M. Queff\'elec, Substitution dynamical systems |
  %spectral analysis, \textit{Lecture Notes in Mathematics 1294},
  %Springer, 1987.

%\bibitem[Urb86]{Urb86} M. Urba\'nski, On Hausdorff dimension of invariant sets for expanding maps of a circle, \textit{Ergodic Th, \& Dynam. Sys.} 6 (1986), 295-309.

%\bibitem[Wal82]{Wal82} P. Walters, An Introduction to Ergodic Theory, \textit{Graduate Texts in Mathematics 79}, Springer, 1982.

\end{thebibliography}
\end{document}